\documentclass[12pt]{amsart}
\usepackage[T1]{fontenc}
\usepackage{mathtools}
\usepackage{xcolor}
\usepackage{amsmath,amssymb}
\usepackage{amsthm}
\usepackage[thmtools-compat]{keytheorems}
\usepackage{bm}
\usepackage{tikz}
\usepackage{url}
\usepackage[backend=biber,style=alphabetic, maxbibnames=99, maxalphanames=6, url=false, doi=false]{biblatex}
\AtBeginBibliography{\sloppy}

\newcommand{\titletext}{\texorpdfstring{\(g\)}{g}-Positivity for Paving Matroids}
\usepackage{geometry}
\usepackage{hyperref}
\hypersetup{
  hidelinks,
  colorlinks=true,
  linkcolor=blue,
  citecolor=magenta,
  pdftitle={\titletext},
  pdfauthor={Yiyu Wang}
}
\usepackage{cleveref}

\usepackage{newpxtext}

\title{\titletext}
\author{Yiyu Wang}
\address{Department of Mathematics, The Ohio State University, 231 W. 18th Ave., Columbus, OH 43210}
\email{wang.20315@osu.edu}
\date{\today}

\declaretheorem[name=Theorem, numberwithin=section, refname={Theorem,Theorems}]{theorem}
\declaretheorem[name=Lemma, sibling=theorem, refname={Lemma,Lemmas}]{lemma}
\declaretheorem[name=Proposition, sibling=theorem, refname={Proposition,Propositions}]{proposition}
\declaretheorem[name=Corollary, sibling=theorem, refname={Corollary,Corollaries}]{corollary}
\declaretheorem[name=Conjecture, sibling=theorem, refname={Conjecture,Conjectures}]{conjecture}

\declaretheorem[name=Definition, style=definition, sibling=theorem, refname={Definition,Definitions}]{definition}

\numberwithin{equation}{section}

\crefname{section}{Section}{Sections}
\crefname{subsection}{Subsection}{Subsections}
\crefname{equation}{Equation}{Equations}

\newcommand{\R}{{\mathbb{R}}}

\newcommand{\Z}{{\mathbb{Z}}}

\begin{document}

\begin{abstract}
  We prove that the \(g\)-polynomial of every paving matroid has nonnegative coefficients. Our proof combines a deletion-contraction argument with an elementary coefficient inequality for the \(g\)-polynomials of uniform matroids.
\end{abstract}

\maketitle

\section{Introduction}

\subsection{Overview}
The hypersimplex \(\Delta_{r,n}\), the convex hull of the \(0/1\)-vectors in \(\R^n\) with exactly \(r\) ones, is the base polytope of the uniform matroid \(U_{r,n}\), and the base polytope of every matroid of rank \(r\) on \(n\) elements is a subpolytope of it. Subdivisions of \(\Delta_{r,n}\) into matroid base polytopes are fundamental objects in tropical geometry, where the regular ones correspond to tropical linear spaces \cite{speyer2008tropical}. Speyer proposed the following bound on the \(f\)-vectors of such subdivisions.

Throughout, we use the convention that \(\binom{a}{b}=0\) unless \(0\leq b\leq a\).

\begin{conjecture}[\(f\)-vector conjecture, \cite{speyer2008tropical}]
  Let \(\mathcal{S}\) be a subdivision of \(\Delta_{r,n}\) into matroid base polytopes, and for \(1\leq i\leq n\) let \(f_i\) be the number of cells of \(\mathcal{S}\) of dimension \(n-i\) that are not contained in the boundary of \(\Delta_{r,n}\). Then
  \[
    f_i\leq \binom{n-i-1}{r-i}\binom{n-r-1}{i-1}.
  \]
\end{conjecture}

To approach this conjecture, Speyer \cite{speyer2009matroid} introduced a matroid invariant \(g_M(t)\in\Z[t]\), defined through the \(K\)-theory of the Grassmannian for matroids representable over a field of characteristic zero. Later, it was extended to all matroids by Fink and Speyer \cite{fink2012kclasses}, and alternative descriptions were given in \cite{BergetEurSpinkTseng2023,ferroni2024valuative}. For our purposes, the following equivalent characterization is convenient.

\begin{definition}\label{def:g-poly}
  There is a unique way of associating to each matroid \(M\) a polynomial \(g_M(t)\) with integer coefficients, in such a way that the following properties hold:
  \begin{enumerate}
    \item If \(M\) has loops or coloops, \(g_M(t)=0\).
    \item \(g_M(t)\) is invariant under series and parallel extensions (see \cite[Corollary~9.4]{fink2012kclasses}).
    \item If \(M=M_1\oplus M_2\), then \(g_M(t)=g_{M_1}(t)g_{M_2}(t)\) (see \cite[Proposition~9.17(a)]{ferroni2024valuative}).
    \item \(M\to g_M(t)\) is covaluative.
    \item \(g_{U_{1,2}}(t)=t\).
  \end{enumerate}
  Here \(U_{1,2}\) means the uniform matroid of rank one on a two-element set. In (2), we exclude series extensions at loops and parallel extensions at coloops.
\end{definition}
Uniqueness follows from \cite[Theorem~9.15]{ferroni2024valuative}, since properties (2), (3) and (5) imply \(g_M(t)=t^c\) whenever \(M\) is a loopless and coloopless direct sum of \(c\) series-parallel matroids.

Speyer showed that the \(f\)-vector conjecture would follow from the following conjecture.

\begin{conjecture}[\(g\)-positivity conjecture, \cite{speyer2009matroid}]
  For every matroid \(M\), the polynomial \(g_M(t)\) has nonnegative coefficients.
\end{conjecture}

Speyer proved the \(g\)-positivity conjecture for matroids representable over a field of characteristic zero \cite[Proposition~3.3]{speyer2009matroid}, and Ferroni and Schröter proved it for sparse paving matroids \cite[Theorem~9.21]{ferroni2024valuative}. Fink, Shaw and Speyer \cite{fink2026omega} studied \(\omega(M)\), the coefficient of \(t^{r(M)}\) in \(g_M(t)\), and proved that \(g_M(t)\) has nonnegative coefficients if \(M/F\) is connected for every proper flat \(F\) of \(M\), assuming \(\omega(M)\geq 0\) for every matroid \(M\) \cite[Theorem~1.5]{fink2026omega}. Berget and Fink proved the nonnegativity of \(\omega(M)\) \cite[Theorem~E]{berget2025externalactivitycomplexpair}; another proof was given by Eur, Fink and Larson \cite[Corollary~1.2]{eur2025vanishingtheoremscombinatorialgeometries}. Consequently, the \(g\)-positivity conjecture holds for every matroid of rank at most three \cite[Corollary~5.39]{larson2026aspects}. It remains open in rank four and higher.

\subsection{Main result}
A matroid of rank \(r\) is \emph{paving} if all of its circuits have at least \(r\) elements. Paving matroids form a large and well-studied class of matroids. It is conjectured that asymptotically almost all matroids are sparse paving \cite{mayhew2011asymptotic}. In this note we prove the \(g\)-positivity conjecture for paving matroids.

\begin{theorem}\label{thm:main}
  For every connected paving matroid \(M\), \(g_M(t)\in \Z_{\geq 0}[t]\).
\end{theorem}

The connectedness assumption does no harm due to \cref{def:g-poly}.
The proof is by induction on the number of elements via deletion and contraction, starting from the formula of Ferroni and Schröter for the \(g\)-polynomial of a paving matroid \cite[Theorem~9.18]{ferroni2024valuative}. The key step, \cref{lem:key}, is an elementary statement about the \(g\)-polynomials of uniform matroids.

\section{A Deletion-Contraction Type Argument}\label{sec:proof}

Throughout this section, let \(M\) be a rank \(r\geq 3\) connected paving matroid on \(E=[n]\). We will try to relate \(g_M\) to \(g_{M/e}\) and \(g_{M\setminus e}\). We first need a few lemmas on the connectedness of the minors.

\begin{lemma}\label{lem:contraction-connected}
  Let \(M\) be a connected simple matroid. Then there is an \(e\in E\) such that \(M/e\) is connected.
\end{lemma}
\begin{proof}
  Suppose that \(M/e\) is disconnected for every \(e\in E\). Then \(M^*\) is a minimally connected matroid. By \cite[Proposition~4.3.9]{oxley2011matroid}, every minimally connected matroid has a size-two cocircuit, so \(M\) has a size \(2\) circuit, contradicting simplicity.
\end{proof}

\begin{lemma}\label{lem:series-extension}
  Let \(M\) be a connected paving matroid of rank \(r\geq 3\), and suppose that \(e\in E\) is such that \(M\setminus e\) is disconnected. Then \(M\) is a series extension of \(M/e\).
\end{lemma}
\begin{proof}
  We first show that \(M\setminus e\) has a coloop \(f\). Since \(M\setminus e\) is disconnected, we can write \(M\setminus e=M|A\oplus M|B\), \(A\sqcup B=E\setminus e\). Let \(a=r(A)\) and \(b=r(B)\). Since \(r\geq 3\) and \(M\) is connected, \(M\) is both loopless and coloopless. Then
  \begin{itemize}
    \item \(a+b=r\), since \(e\) is not a coloop.
    \item \(a,b\geq 1\) since \(M\) is loopless.
  \end{itemize}
  Suppose \(M|A\) is coloopless. Then every element of \(A\) is in some circuit \(C\) of \(M|A\). Since \(M\) is paving, \(r\leq |C|\leq r(C)+1\leq a+1\). Hence \(a\geq r-1\) and \(b=r-a\leq 1\), so \(b=1\).

  Similarly, if \(M|B\) is also coloopless, then \(a=1\), so \(r=a+b=2\), contradicting \(r\geq 3\). We conclude that \(M\setminus e\) must have a coloop \(f\).

  Since \(f\) is not a coloop of \(M\) but is a coloop of \(M\setminus e\), we have \(r_{M^*}(\{e,f\})=r_{M^*}(\{e\})\). Hence there is a cocircuit contained in \(\{e,f\}\). Since \(M\) is coloopless, \(\{e,f\}\) is a cocircuit, and hence \(M\) is a series extension of \(M/e\).
\end{proof}

\begin{corollary}\label{cor:deletion-disconnected}
  If \(M\setminus e\) is disconnected, then \(g_M(t)=g_{M/e}(t)\).
\end{corollary}
\begin{proof}
  By \cref{lem:series-extension}, \(M\) is a series extension of \(M/e\), and the \(g\)-polynomial is invariant under series extensions by \cref{def:g-poly}.
\end{proof}

Let \(\mathcal{H}\) be the set of hyperplanes of \(M\) (i.e., the flats of corank one) of size at least \(r\). For integers \(2\leq r\leq m\), set
\[
  p_{r,m}(t)=g_{U_{r,m}}(t)+t\,g_{U_{r-1,m-1}}(t).
\]
In particular, \(p_{r,r}(t)=0\). Recall from \cite[Proposition~9.16]{ferroni2024valuative} that for \(1\leq r<m\),
\begin{equation}\label{eq:uniform}
  g_{U_{r,m}}(t)=\sum_{k=1}^{r} \binom{m-k-1}{r-k} \binom{m-r-1}{k-1}t^k.
\end{equation}

We need the following elementary lemma.

\begin{lemma}\label{lem:recursion}
  For integers \(r\geq 2\) and \(m\geq r\), we have
  \begin{equation}\label{eq:uniform-recursion}
    g_{U_{r,m+1}}(t)-g_{U_{r,m}}(t) =g_{U_{r-1,m}}(t)+t\,g_{U_{r-1,m-1}}(t),
  \end{equation}
  and for integers \(r\geq 3\) and \(m\geq r\), we have
  \begin{equation}\label{eq:p-recursion}
    p_{r,m+1}(t)-p_{r,m}(t)=p_{r-1,m}(t)+t\,p_{r-1,m-1}(t).
  \end{equation}
\end{lemma}
\begin{proof}
  If \(m=r\), both sides of \eqref{eq:uniform-recursion} equal \(t\), by \eqref{eq:uniform}. Let \(m\geq r+1\) and \(1\leq k\leq r\). By \eqref{eq:uniform} and Pascal's rule, the coefficient of \(t^k\) on the left-hand side of \eqref{eq:uniform-recursion} is
  \begin{align*}
    &\binom{m-k}{r-k}\binom{m-r}{k-1}-\binom{m-k-1}{r-k}\binom{m-r-1}{k-1}\\
    &\qquad=\binom{m-k-1}{r-k-1}\binom{m-r}{k-1}+\binom{m-k-1}{r-k}\left[\binom{m-r}{k-1}-\binom{m-r-1}{k-1}\right]\\
    &\qquad=\binom{m-k-1}{r-k-1}\binom{m-r}{k-1}+\binom{m-k-1}{r-k}\binom{m-r-1}{k-2},
  \end{align*}
  which is the coefficient of \(t^k\) in \(g_{U_{r-1,m}}(t)+t\,g_{U_{r-1,m-1}}(t)\) by \eqref{eq:uniform}. Since both sides of \eqref{eq:uniform-recursion} have degree at most \(r\) and zero constant term, \eqref{eq:uniform-recursion} follows.

  \eqref{eq:p-recursion} can be obtained by applying \eqref{eq:uniform-recursion} to \((r,m)\) and to \((r-1,m-1)\).
\end{proof}

\begin{proposition}\label{prop:deletion-contraction}
  Suppose that \(e\in E\) is such that both \(M/e\) and \(M\setminus e\) are connected. Then
  \[
    g_M(t)=g_{M/e}(t)+g_{M\setminus e}(t)+t\left(g_{U_{r-1,n-2}}(t)-\sum_{e\in H}p_{r-1,|H|-1}(t)\right),
  \]
  where \(H\) ranges over \(\mathcal{H}\).
\end{proposition}
\begin{proof}
  By \cite[Theorem~9.18]{ferroni2024valuative}, for a connected paving matroid we have
  \begin{equation}\label{eq:FS}
    g_M(t)=g_{U_{r,n}}(t)- \sum_{H\in\mathcal H}p_{r,|H|+1}(t).
  \end{equation}
  Since \(M/e\) and \(M\setminus e\) are both connected paving matroids,
  the same formula applies to \(M/e\) and \(M\setminus e\). The hyperplanes of \(M/e\) of size at least \(r-1\) are exactly the sets \(H\setminus e\) with \(H\in\mathcal{H}\) and \(e\in H\), hence
  \begin{equation}\label{eq:contraction}
    g_{M/e}(t)=g_{U_{r-1,n-1}}(t)- \sum_{e\in H} p_{r-1,|H|}(t).
  \end{equation}

  A hyperplane of \(M\setminus e\) of size at least \(r\) corresponds uniquely to either a hyperplane \(H\in\mathcal{H}\) with \(e\notin H\), or a set \(H\setminus e\) with \(H\in\mathcal{H}\), \(e\in H\) and \(|H|\geq r+1\). Indeed, if \(F\) is a hyperplane of \(M\backslash e\) with \(|F|\geq r\), then the closure of \(F\) in \(M\) is either \(F\) or \(F\cup \{e\}\), corresponding to the two cases described. Conversely, suppose \(H\in \mathcal{H}\). If \(e\not\in H\), then it is a hyperplane of \(M\backslash e\) and \(|H|\geq r\). If \(e\in H\) and \(|H|\geq r+1\), we claim that \(H\setminus e\) is a hyperplane of \(M\backslash e\). If not, \(r(H\setminus e)=r-2\), \(|H\setminus e|\geq r\), therefore \(H\setminus e\) contains a circuit of size \(\leq r-1\), contradicting pavingness of \(M\backslash e\). Thus \(H\setminus e\) is a hyperplane of \(M\backslash e\) of size at least \(r\).

  Since \(p_{r,r}(t)=0\), the hyperplanes \(H\in \mathcal{H}\) with \(e\in H\) and \(|H|=r\) contribute zero, hence
  \begin{equation}\label{eq:deletion}
    g_{M\setminus e}(t) =g_{U_{r,n-1}}(t) -\sum_{e\not\in H}p_{r,|H|+1}(t) -\sum_{e\in H}p_{r,|H|}(t).
  \end{equation}
  Combining \cref{eq:FS,eq:contraction,eq:deletion} with \cref{eq:uniform-recursion} (for \(m=n-1\)) and \cref{eq:p-recursion} (for \(m=|H|\), where \(H\in\mathcal{H}\) and \(e\in H\)), we obtain the desired identity.
\end{proof}

The key lemma is the following.

\begin{lemma}\label{lem:key}
  Let \(r\geq 2\) and \(m\geq r+2\) be integers, and let \(h_1,\ldots,h_s\) be integers satisfying \(r+1\leq h_j\leq m-1\) for each \(j\). If
  \[
    g_{U_{r,m}}(t)-\sum_j p_{r,h_j}(t)\in \Z_{\geq 0}[t],
  \]
  then
  \[
    g_{U_{r,m-1}}(t)-\sum_j p_{r,h_j-1}(t)\in \Z_{\geq 0}[t].
  \]
\end{lemma}
\begin{proof}
  Write \([t^k]f(t)\) for the coefficient of \(t^k\) in \(f(t)\), and fix \(k\geq 1\). By \eqref{eq:uniform}, the polynomials \(p_{r,h_j}(t)\) and \(p_{r,h_j-1}(t)\) have nonnegative coefficients, and \([t^k]p_{r,h_j-1}(t)=0\) unless \(k\leq r\) and \(h_j\geq r+k\). If \(k>r\), or if \(h_j<r+k\) for all \(j\), there is nothing to prove. Otherwise, discarding the \(j\) with \(h_j<r+k\), we may assume that \(k\leq r\) and \(r+k\leq h_j\leq m-1\) for all \(j\). In particular, \(m\geq r+k+1\). Then
  \[
    \frac{[t^k]g_{U_{r,m-1}}(t)}{[t^k]g_{U_{r,m}}(t)}=\frac{m-r-k}{m-k-1},
  \]
  and a direct calculation using the definition of \(p_{r,h_j}(t)\) gives
  \[
    \frac{[t^k]p_{r,h_j-1}(t)}{[t^k]p_{r,h_j}(t)}=\frac{h_j-r-1}{h_j-r}\cdot\frac{h_j-r-k+1}{h_j-k-1}=\Bigl(1-\frac{1}{h_j-r}\Bigr)\Bigl(1-\frac{r-2}{h_j-k-1}\Bigr).
  \]
  Both factors are nonnegative and nondecreasing in \(h_j\geq r+k\), as \(r\geq 2\). Since \(h_j\leq m-1\), we get
  \[
    \frac{[t^k]p_{r,h_j-1}(t)}{[t^k]p_{r,h_j}(t)}\leq \frac{m-r-2}{m-r-1}\cdot\frac{m-r-k}{m-k-2}\leq \frac{m-r-k}{m-k-1},
  \]
  where the second inequality amounts to \(\frac{m-r-2}{m-r-1}\leq \frac{m-k-2}{m-k-1}\), which holds when \(k\leq r\). Summing over \(j\) and using the hypothesis, we conclude that
  \[
    \sum_j [t^k]p_{r,h_j-1}(t)\leq \frac{m-r-k}{m-k-1}\sum_j [t^k]p_{r,h_j}(t)\leq \frac{m-r-k}{m-k-1}\,[t^k]g_{U_{r,m}}(t)=[t^k]g_{U_{r,m-1}}(t).\qedhere
  \]
\end{proof}

\begin{proof}[Proof of \cref{thm:main}]
  We induct on \(|E|\) over connected paving matroids of all ranks. The case \(r\leq 1\) follows from \cref{def:g-poly}. If \(r=2\), connectedness implies that the simplification of \(M\) is \(U_{2,m}\) for some \(m\geq 3\). By parallel-extension invariance and \eqref{eq:uniform},
  \[
    g_M(t)=g_{U_{2,m}}(t)=(m-2)t+(m-3)t^2\in\Z_{\geq 0}[t].
  \]

  Let \(r\geq 3\). Since \(M\) is paving, \(M\) is simple. Since \(M\) is connected, \(n\geq r+1\). If \(n=r+1\), then \(M\cong U_{r,r+1}\) and \(g_M(t)=t\) by \eqref{eq:uniform}. Hence we may assume that \(n\geq r+2\).

  By \cref{lem:contraction-connected}, we can find an \(e\in E\) such that \(M/e\) is connected. If \(M\setminus e\) is disconnected, then \(g_M(t)=g_{M/e}(t)\in \Z_{\geq 0}[t]\) by \cref{cor:deletion-disconnected}, since \(M/e\) has \(|E|-1\) elements.

  If both \(M/e\) and \(M\backslash e\) are connected, we can use \cref{prop:deletion-contraction}. By the induction hypothesis, both \(g_{M/e}(t)\) and \(g_{M\setminus e}(t)\) have nonnegative coefficients. We only need to show the polynomial in the parentheses has nonnegative coefficients. Since \(g_{M/e}(t)\) has nonnegative coefficients, we have
  \[
    g_{U_{r-1,n-1}}(t)-\sum_{e\in H} p_{r-1,|H|}(t)\in \Z_{\geq 0}[t].
  \]
  Since \(r-1\geq 2\), \(n-1\geq (r-1)+2\), and \(r\leq |H|\leq n-2\) for every \(H\in\mathcal{H}\) (as \(M\) has no coloops), \cref{lem:key} applies and gives
  \[
    g_{U_{r-1,n-2}}(t)-\sum_{e\in H} p_{r-1,|H|-1}(t)\in \Z_{\geq 0}[t],
  \]
  and therefore \(g_M(t)\in \Z_{\geq 0}[t]\).
\end{proof}

\printbibliography

\end{document}